\pdfoutput=1
\RequirePackage{ifpdf}
\ifpdf 
\documentclass[pdftex]{sigma}
\else
\documentclass{sigma}
\fi

\numberwithin{equation}{section}

\newtheorem{Theorem}{Theorem}[section]
{ \theoremstyle{definition}
\newtheorem{Definition}[Theorem]{Definition}
\newtheorem{Example}[Theorem]{Example}
\newtheorem{Remark}[Theorem]{Remark} }

\newcommand{\bigslant}[2]{{\raisebox{.2em}{$#1$}\big/\raisebox{-.2em}{$#2$}}}

\begin{document}

\allowdisplaybreaks

\newcommand{\arXivNumber}{1703.08279}

\renewcommand{\PaperNumber}{029}

\FirstPageHeading

\ShortArticleName{Symplectic Structures in Frame Bundles and Symplectic Cohomology}

\ArticleName{On the Symplectic Structures in Frame Bundles\\ and the Finite Dimension of Basic Symplectic\\ Cohomologies}

\Author{Andrzej CZARNECKI}

\AuthorNameForHeading{A.~Czarnecki}

\Address{Jagiellonian University, {\L}ojasiewicza 6, 30-348 Krakow, Poland}
\Email{\href{mailto:andrzejczarnecki01@gmail.com}{andrzejczarnecki01@gmail.com}}

\ArticleDates{Received February 16, 2018, in final form March 24, 2018; Published online March 30, 2018}

\Abstract{We present a construction (and classification) of certain invariant 2-forms on the real symplectic group. They are used to define a symplectic form on the quotient by a~maximal torus and to ``lift'' a symplectic structure from a~symplectic manifold to the bundle of frames. This is a by-product of a failed attempt to prove certain finiteness theorems for basic symplectic cohomologies. In the last part of the paper we include a valid proof.}


\Classification{53C12; 57R18}

\section{Introduction}
The first part of this paper focuses on the obstacles arising in the proof of a foliated version of the main results on symplectic cohomology from \cite{tseng}. This should be viewed as a contribution to the discussion about the ``correct'' symplectic Hodge theory, a well established field of study, cf.~\cite{brylinski,cavalcanti,fernandezibanez,yau,yan}. The foliated analogues (cf.~\cite{baczekija,pawelija,he,pak}) help to better understand the limitations of the theories. For example, there is always a compatible metric for a symplectic structure on a manifold, whereas there need not be one for a transverse symplectic structure of a foliation, and in the present paper we highlight some consequences it has for the symplectic cohomologies defined in~\cite{tseng}.

It is well known that the basic cohomology of a Riemannian foliation is finite dimensional. There are in principle two ways of proving that: one is the Kamber--Tondeur groupoid-invariant version of elliptic analysis, cf.~\cite{azizeliptyczny}; the other is the Molino's structure theorems and relations between different cohomologies in appropriate spectral sequences of certain fiber bundles, cf.~\cite{aziztopologiczny}.

In the second part of this paper we show that the Kamber--Tondeur analysis applies to the proof of the following theorem.

\begin{Theorem}[finiteness theorem] \label{main} Basic symplectic cohomology of a transversally symplectic Riemannian foliation is finite dimensional.
\end{Theorem}

The relevant cohomologies are due to Tseng and Yau and will be defined later. Our initial feeling had been that the metric condition is superfluous. In the proof presented in Section \ref{wreszciedowod}, the Riemannian structure is indeed used to establish elliptic theory and to give some compactness constraints on the objects considered, but the symplectic cohomologies seem to be more involved with the algebraic structure of a~certain $\mathfrak{sl}(2,\mathbb{R})$-representation that does not depend on compactness of the space (cf.~\cite{yan} and \cite[Section~3.2.1]{tseng}). However, in \cite{pawelija} we recently gave an example of a very simple non-Riemannian transversally symplectic foliation, which shows that there is no possibility to extend Theorem~\ref{main} to all transversally symplectic foliations, or indeed to any other interesting broad class. We include it in the last section, in Example~\ref{przy}.

\looseness=-1 Having never developed a taste for elliptic theory, we initially tried to apply Molino's construction to the symplectic case. As will be explained in Section \ref{bundle}, it turns out that the appropriate spectral sequence comparison needs not work for the symplectic cohomologies but the relevant fiber bundle can be (somewhat surprisingly) defined. Thus we stumbled upon a construction that seems to be of independent interest. We will define a 2-form on the symplectic group as close to an invariant symplectic form as it can possibly get, and proceed to define a ``lifting'' similar to the canonical lifting of Molino. The need to use inverted commas will be explained in Remarks~\ref{dodac} and~\ref{dodacdwa}. We hope that the algebraic discussion of $\mathfrak{sp}(2n,\mathbb{R})$ will be appealing as well.

We also hope that the reader will accept (or forgive us) our exposition: we will describe withouth details the Riemannian prototype before stating our results in the symplectic setting. We feel it is appropriate to highlight in this way the differences in the mechanisms behind the metric and symplectic canonical lifts and Hodge theories.

The paper ends with the proof of Theorem \ref{main} and several remarks and examples.

\section{Preliminaries} \label{preliminaries}

We fix notation and recall some of the results about transverse structures and symplectic cohomologies.

First of all, in this paper we often switch between ``foliated'' and ``non-foliated'' cases. Much of the foliated constructions and examples is treated briefly as we try to focus on the topics interesting to the broader audience. At the same time, we include all the relevant foliated definitions and properties for the convenience of a non-foliated reader.

While we mostly work with a compact symplectic manifold $(M,\omega)$, we also distinguish ``non-compact'' and ``compact'' cases. The former pertains to the frame bundles with the non-compact fiber ${\rm Sp}(2n,\mathbb{R})$, the latter~-- to the ${\rm U}(n)$- or ${\rm SO}(2n)$-frame bundles. Throughout we write $\operatorname{Sp} M$, $\operatorname{U} M$, and $\operatorname{SO} M$ for these bundles, respectively. Only in Example~\ref{disc} the initial symplectic manifold will be non-compact.
\begin{Definition} \label{fol}
A \emph{codimension} $n$ \emph{foliation} $\mathcal{F}$ on a smooth manifold~$M$ is given by the following data:
\begin{itemize}\itemsep=0pt
\item an open cover $\mathcal{U}:=\{U_i\}_{i\in I}$ of $M$;
\item a $n$-dimensional smooth manifold $T_0$;
\item for each $U_i\in\mathcal{U}$ a submersion $f_i\colon U_i\longrightarrow T_0$ with connected fibers called \emph{plaques};
\item for each intersection $U_i\cap U_j\neq\varnothing$ a local diffeomorphism $\gamma_{ij}$ of $T_0$ such that $f_j=\gamma_{ij}\circ f_i$.
\end{itemize}
The last condition ensures that the plaques glue nicely to form a partition of $M$ by submanifolds of codimension $n$ called the \emph{leaves} of~$\mathcal{F}$. The dimension of these submanifolds is called the \emph{dimension} of the foliation.
\end{Definition}

We call $T=\coprod\limits_{U_i\in\mathcal{U}}f_i(U_i)$ the \emph{transverse manifold} of $\mathcal{F}$. The local diffeomorphisms $\gamma_{ij}$ generate a pseudogroup $\Gamma$ of transformations of~$T$ (called the \emph{holonomy pseudogroup}). The space of leaves $\bigslant{M}{\mathcal{F}}$ of the foliation $\mathcal{F}$ can be identified with $\bigslant{T}{\Gamma}$. We note that neither $T$ nor $\bigslant{T}{\Gamma}$ need to be compact, even if~$M$ is.
\begin{Definition}
A smooth form $\omega$ on $M$ is called \emph{basic} if for any vector field $X$ tangent to the leaves of $\mathcal{F}$ we have
\begin{gather*}
i_X\omega= i_X{\rm d}\omega= 0.
\end{gather*}
Basic forms are in one-to-one correspondence with $\Gamma$-invariant smooth forms on $T$.
\end{Definition}

It is clear that $d\omega$ is basic for any basic form $\omega$. Hence the set of basic forms of $\mathcal{F}$, $\Omega^{\bullet}(M\slash\mathcal{F})$, is a subcomplex of the de Rham complex of~$M$. We define the \emph{basic cohomology} (or sometimes \emph{basic de Rham cohomology} if other cohomologies of basic forms will be in play) of $\mathcal{F}$ to be the cohomology of this subcomplex and denote it by $H^{\bullet}(M\slash\mathcal{F})$.

We note that every manifold is foliated by points (with codimension equal to the dimension of the manifold) and then all forms are basic and the basic cohomology becomes ordinary cohomology. This remark applies to any structure below, thus any non-foliated case can be viewed as the special instance of foliated case.

We will adapt the usual Hodge star operator to the foliated situation. We will define it to preserve the basic forms. Consider a \emph{Riemannian foliation} $(M,\mathcal{F},g)$, i.e., a manifold endowed with a Riemannian metric $g$ for which $\mathcal{L}_Xg|_{T\mathcal{F}^{\bot}}=0$ if $X$ is tangent to the foliation~$\mathcal{F}$. Clearly, this defines a $\Gamma$-invariant metric on the transverse manifold. The volume form associated with~$g$ is of the form $\text{dvol} =\theta\wedge \chi$, where $\theta$ is a basic non-degenerate $n$-form (a~$\Gamma$-invariant volume form on~$T$). We can define the (metric) star pointwise: locally, for a non-vanishing basic $k$-form~$A$, $\star_{m} A$ is the unique basic $(n-k)$-form such that in each point~$x$
\begin{gather*}
(A\wedge\star_{m} A)_x=||A||^2_x\theta_x
\end{gather*}
(with the usual norm $||\cdot||$ induced pointwise by $g$) and we extend this definition invoking \smash{$\mathcal{C}^{\infty}(M\slash \mathcal{F})$}-linearity. We can then define a scalar product on $\Omega^{\bullet}(M\slash\mathcal{F})$ by
\begin{gather*}
\langle A,B\rangle = \int_{M} A\wedge \star_{m} B\wedge \chi
\end{gather*}
using again the associated volume to integrate.

From now on, we will assume that our foliation admits a closed basic 2-form of maximal rank, $\omega$. It corresponds to a $\Gamma$-invariant symplectic structure on~$T$ and we call it a \emph{transverse symplectic structure} on $\mathcal{F}$. Observe that the codimension must be even,~$2n$. In parallel with the Riemmanian structure~$\omega$ defines a Hodge star operator, but (due to its antisymmetry) in a~slightly different manner. Following \cite[Section~2.1]{brylinski}, there is a non-degenerate pairing of vector fields
\begin{gather*}
X, Y \mapsto \omega(X,Y)
\end{gather*}
that induces a non-degenerate pairing of $k$-forms
\begin{gather*}
\alpha, \beta \mapsto G(\alpha,\beta)
\end{gather*}
that in turn enables us to define an isomorphism
\begin{gather*}
\star_{s}\colon \ \Omega^{k}(T)\longrightarrow \Omega^{2n-k}(T)
\end{gather*}
by the following condition: if $\alpha$ and $\beta$ are $k$-forms, then
\begin{gather*}
\alpha\wedge \star_s\beta=G(\alpha,\beta)\omega^n.
\end{gather*}
It can be easily seen that $\Gamma$-invariance of $\omega$ implies that $\star_s$ preserves $\Gamma$-invariant forms, and so do the operators described in the following definition.
\begin{Definition}
The operators
\begin{gather*}
{\rm d}^{\Lambda}:=\star_s {\rm d} \star_s, \qquad {\rm d}+{\rm d}^{\Lambda},\qquad {\rm d}{\rm d}^{\Lambda}
\end{gather*}
preserve the basic forms. It is well known that ${\rm d}{\rm d}^{\Lambda}+{\rm d}^{\Lambda}{\rm d}=0$ and thus we can define the \emph{basic symplectic cohomologies}
\begin{gather*}
H^{\bullet}_{{\rm d}+{\rm d}^{\Lambda}}(M\slash\mathcal{F}):=\frac{\operatorname{Ker}({\rm d}+{\rm d}^{\Lambda})\cap \Omega^{\bullet}(M\slash\mathcal{F})}{\operatorname{Im}({\rm d}{\rm d}^{\Lambda})\cap \Omega^{\bullet}(M\slash\mathcal{F})},\\
H^{\bullet}_{{\rm d}{\rm d}^{\Lambda}}(M\slash\mathcal{F}):=\frac{\operatorname{Ker}({\rm d}{\rm d}^{\Lambda})\cap\Omega^{\bullet}(M\slash\mathcal{F})}{\operatorname{Im}({\rm d})+\operatorname{Im}({\rm d}^{\Lambda})\cap \Omega^{\bullet}(M\slash\mathcal{F})}.
\end{gather*}
We call ${\rm d}^{\Lambda}$ a \emph{codifferential} and sometimes refer to $H^{\bullet}_{{\rm d}+{\rm d}^{\Lambda}}(M\slash\mathcal{F})$ as to \emph{$({\rm d}+{\rm d}^{\Lambda})$-cohomology}. We note that the codifferential is of degree~$-1$.
\end{Definition}

We refer the interested reader to \cite{tseng,yan}, and \cite{baczekija,pawelija} for a more in-depth treatment of the topic in the non-foliated and foliated cases, respectively.

We point out that $H^{\bullet}_{{\rm d}+{\rm d}^{\Lambda}}(M\slash\mathcal{F})$ and $H^{\bullet}_{{\rm d}{\rm d}^{\Lambda}}(M\slash\mathcal{F})$ are both isomorphic to the basic de Rham cohomology for transversally K\"ahler foliations by the arguments of~\cite{pawel} or~\cite{baczekija}. Thus, as in the non-foliated case, these symplectic cohomologies measure how far from being K\"ahler the symplectic structure is.

We want to establish Theorem~\ref{main} for these symplectic cohomologies and we will do so in the last chapter, combining~\cite{tseng} with the elliptic methods from~\cite{azizeliptyczny}. First, however, we will focus on the Molino's construction mentioned above.

\section{``Canonical lift'' of a transversally symplectic foliation} \label{bundle}

We would like to try to repeat the following construction due to El Kacimi-Alaoui, Hector, and Sergiescu.
\begin{Theorem}[\cite{aziztopologiczny}] \label{aziz}
If $(M,\mathcal{F},g)$ is a Riemannian foliation, then
\begin{itemize}\itemsep=0pt
\item there exists a \emph{canonical lifting} $($cf.~{\rm \cite[Proposition~2.4]{molino})} to a foliation $\big({\rm SO}_{\mathcal{F}}M,\mathcal{F}^{{\rm SO}}\big)$ of the same dimension, where ${\rm SO}_{\mathcal{F}}M$ is the bundle of orthonormal frames transverse to $\mathcal{F}$ over $M$; its projection becomes a foliated map; the foliation $\mathcal{F}^{{\rm SO}}$ is again Riemannian;
\item closures of leaves of $\mathcal{F}^{{\rm SO}}$ are the fibers of a submersion from ${\rm SO}_{\mathcal{F}}M$ to a compact mani\-fold~$W$;
\item the basic cohomology of $\big({\rm SO}_{\mathcal{F}}M,\mathcal{F}^{{\rm SO}}\big)$ is isomorphic to the tensor product of de Rham cohomology of $W$ and a certain Lie algebra, hence it is finite dimensional;
\item the basic cohomology of $(M,\mathcal{F})$ sits naturally on a second page of a spectral sequence converging to ${\rm SO}(n)$-invariant basic cohomology of $\big({\rm SO}_{\mathcal{F}}M,\mathcal{F}^{{\rm SO}}\big)$; since ${\rm SO}(n)$ is compact and connected, this is isomorphic to $H^{\bullet}\big({\rm SO}_{\mathcal{F}}M,\mathcal{F}^{{\rm SO}}\big)$, hence finite dimensional; it can be then shown that the second page of the sequence is finite dimensional as well, establishing the finite dimension of $H^{\bullet}(M\slash\mathcal{F})$.
\end{itemize}
\end{Theorem}
There are several problems to be encountered in the transversally symplectic case, the first of which is to construct the canonical lifting. In the Riemannian case, it is locally a product Riemannian foliation $\big(U,\mathcal{F}|_U,g|_U\big)\times \big({\rm SO}(n),\ast,m_{{\rm SO}(n)}\big)$ with $m_{{\rm SO}(n)}$ being the unique bi-invariant metric on ${\rm SO}(n)$. In the symplectic case, we need to find an appropriate structure on the fiber ${\rm Sp}(2n,\mathbb{R})$ and before that, to refer the reader again to Remarks~\ref{dodac} and~\ref{dodacdwa} to emphasise that our structure will not be a~local product in the same way.

We now turn to the symplectic geometry of ${\rm Sp}(2n,\mathbb{R})$.

\subsection[2-forms on $\mathfrak{sp}(2n,\mathbb{R})$]{2-forms on $\boldsymbol{\mathfrak{sp}(2n,\mathbb{R})}$}

We promptly make some points. First of all, for even $n$ the symplectic group ${\rm Sp}(2n,\mathbb{R})$ admits a symplectic structure as a non-compact almost complex manifold by a well-known theorem of Gromov. We are however interested in left-invariant symplectic structures. Second, neither ${\rm Sp}(2n,\mathbb{R})$ admits such a structure (by linear algebra below) nor do the two other groups of interest to us, ${\rm U}(n)$ and ${\rm SO}(2n)$ (by their compactness and cohomology). We need to get rid of the inevitable degeneracies of a left-invariant 2-form by taking appropriate quotients. This also solves the problem of possible odd dimension of any of these groups.

Third, we need to stress that the considerations below apply to any semi-simple Lie group, but we refrain from expressing the construction in its maximal generality, instead keeping to the groups arising naturally in the context of symplectic manifolds and foliations.

We fix the notation. Consider the $(2n^2+n)$-dimensional matrix group
\begin{gather*}
{\rm Sp}(2n,\mathbb{R})=\big\{X\in M(2n\times 2n,\mathbb{R})\,|\,X^{\rm t}JX=J\big\}
\end{gather*}
(where $J$ is the standard complex multiplication matrix $\left[\begin{smallmatrix} 0&-I\\ I&0 \end{smallmatrix}\right]$), and its algebra
\begin{gather*}
\mathfrak{sp}(2n,\mathbb{R})=\big\{X\in M(2n\times 2n,\mathbb{R})\,|\,JX=-X^{\rm t}J\big\}
\end{gather*}
and note that the last condition is equivalent to
\begin{gather}
\left\{X=\left[
\begin{matrix}
A & B\\
C & -A^{\rm t}
\end{matrix}
\right]\,|\,A,B,C,D\in M(n\times n,\mathbb{R}),\,B=B^{\rm t},\,C=C^{\rm t}\right\}, \label{(ast)}
\end{gather}
where $L^{\rm t}$ is the transpose of $L$. We begin with the description of the left-invariant closed 2-forms on ${\rm Sp}(2n,\mathbb{R})$ or, equivalently, closed forms on its Lie algebra.
\begin{Theorem}
A closed $2$-form on $\mathfrak{sp}(2n,\mathbb{R})$ has generically rank $2n^2$, i.e., its kernel is of dimension $n$. Moreover, this kernel is generically a commutative Lie algebra.
\end{Theorem}
In particular, there are no forms of higher rank.
\begin{proof}
Choose a closed 2-form $\omega\in\bigwedge^2\mathfrak{sp}(2n,\mathbb{R})^*$. The cohomology of the algebra $\mathfrak{sp}(2n,\mathbb{R})$ is well known, $H^2(\mathfrak{sp}(2n,\mathbb{R})) = H^1(\mathfrak{sp}(2n,\mathbb{R})) = 0$, therefore there exists a unique 1-form $\theta$ such that for all vectors $x$~and~$y$, $\omega(x,y)={\rm d}\theta(x,y)=-\theta([x,y])$. The Killing form $B$ is non-degenerate, and so $\theta(x)=B(a,x)$ for some $a\in\mathfrak{sp}(2n,\mathbb{R})$. The Killing form is invariant, thus $\omega(x,y)=-\theta([x,y])=-B(a,[x,y])=-B([a,x],y)$. We see that the kernel of $\omega$, $\{x\,|\,\omega(x,\cdot)\equiv 0\}$ turns out to be the centralizer of $a$, $\{x\,|\,[a,x]=0\}=:\mathfrak{z}_a$.
We proceed to compute the generic dimension. The set of regular elements (of minimal centralizers) in $\mathfrak{sp}(2n,\mathbb{R})$ consists of matrices diagonalizable over $\mathbb{C}$ with $2n$ distinct complex eigenvalues. A centralizer $\mathfrak{z}_A$ of such an~$A$ is a~$n$-dimensional Cartan algebra of matrices diagonalizable in the same basis. Regularity itself is generic in the Zariski sense, thus also in the usual topology of $\mathbb{R}^{2n^2+n}$.
\end{proof}

For a regular element $A$, we can now define a ``pre-symplectic'' form $\omega_A$ by
\begin{gather*}
\omega_A(x,y):=B(A,[x,y]).
\end{gather*}
Observe that the centralizer of $A$ above is a maximal abelian subalgebra of $\mathfrak{sp}(2n,\mathbb{R})$, therefore a subgroup it is tangent to, $Z_A=\exp \mathfrak{z}_A$, is closed. We will content ourselves with the following theorem as a substitute for the existence of a symplectic structure on ${\rm Sp}(2n,\mathbb{R})$.
\begin{Theorem}\label{struktur}
For a regular element $A\in\mathfrak{sp}(2n,\mathbb{R})$, the homogeneous space $\bigslant{{\rm Sp}(2n,\mathbb{R})}{Z_A}$ is a symplectic non-compact manifold with the symplectic form $\omega^{Z}_{A}$ descending from $\omega_A$ described above.
\end{Theorem}
The notation is meant to indicate that the form $\omega^{G}_A$ is defined by element $A$ on the quotient by its centralizer $G$. We now discuss the choices made along the construction.
\begin{Remark}
First of all we note that the Cartan algebra in the non-compact case is not unique up to conjugation (cf.~\cite{sugiura}). There are however at least two natural choices of the Cartan subgroup, a conjugate copy of a compact torus $\mathbb{T}^n\subset {\rm U}(n)\subset {\rm Sp}(2n,\mathbb{R})$, or the group of the diagonal matrices in ${\rm Sp}(2n,\mathbb{R})$. The latter would carry over more of the algebraic topology of ${\rm Sp}(2n,\mathbb{R})$ to the quotient (since we divide out a contractible group), but we will actually use below the fact that the former kills off the fundamental group. We recall that the maximal torus in a non-compact Lie groups is still unique up to conjugation.
\end{Remark}
\begin{Remark}
We note that contrary to the compact case, Cartan algebras of ${\rm Sp}(2n,\mathbb{R})$ need not even have the minimal dimension among maximal abelian subalgebras. The example of Courter from~\cite{courter} can be fitted into the $A$-part of matrices in~\eqref{(ast)} giving a maximal abelian subalgebra of ${\rm Sp}(28,\mathbb{R})$ of dimension~12, not~14 (the classification from~\cite{masasp} is needed to see that it is indeed maximal). One would be tempted to use such an algebra to retain as much geometry of ${\rm Sp}(2n,\mathbb{R})$ as possible, by keeping the dimension of the quotient high. We prefer Cartan subalgebras for the reason explained above and also for their clear structure (worked out in~\cite{sugiura}, for example), which these smaller algebras lack. We refer the interested reader to the series of papers by Zassenhaus et al.\ \cite{masaupq,masaonc,tenodcourtera,masasp} for a glimpse of the surprising amount of problems in the study of maximal abelian subalgebras.
\end{Remark}
\begin{Remark} Theorem \ref{struktur} paints a slightly more general picture than that of usual compact flag manifolds, well-known to be K\"ahler. In the construction applied to the groups ${\rm SO}(n)$ and~${\rm U}(n)$ (with the inconsequential adjustment that $\theta$ in the proof is now not unique) $Z_A$ is always a~maximal torus and any two choices of this torus (of the subalgebra, of the regular element) therefore give diffeomorphic quotients. In the non-compact case we have less rigidity: consider the example of ${\rm Sp}(2,\mathbb{R})={\rm SL}(2,\mathbb{R})$. It can be divided by either of its (conjugacy classes of) one-dimensional Cartan subgroups: the compact (elliptic) 1-torus to get $\mathbb{R}^2$ or non-compact (hyperbolic) 1-parameter subgroup to get $\mathbb{S}^1\times \mathbb{R}$. We note that the third type of 1-parameter subgroup, non-Cartan and parabolic, would work in this case as well.
\end{Remark}
\begin{Remark} Observe however that even for two maximal tori $Z_A$ and $Z_{A'}$ the diffeomorphic quotients $\big(\bigslant{{\rm Sp}(2n,\mathbb{R})}{Z_A},\omega^{Z}_{A}\big)$ and $\big(\bigslant{{\rm Sp}(2n,\mathbb{R})}{Z_{A'}},\omega^{Z}_{A'}\big)$ may not be symplectomorphic if \smash{$A\neq A'$}, since the Killing form is non-degenerate. There does not seem to be a preferred choice of the element $A$ even in the presence of additional structure, e.g., Riemannian metric.
\end{Remark}

\subsection{Symplectic structures in the frame bundle}

We can now define a ``lifting'' of a symplectic structure to the symplectic frame bundle. We restrict our attention to a compact non-foliated symplectic manifold $(M,\omega)$, we will return to the foliated case in Remark \ref{sedno}. We construct the principal ${\rm Sp}(2n,\mathbb{R})$-bundle of symplectic frames
\begin{gather*}
{\rm Sp}(2n,\mathbb{R}) \to \operatorname{Sp} M \to M
\end{gather*}
with a chosen connection $\nabla$, understood as a choice of the horizontal direction. This bundle is acted upon fiberwise by a maximal torus $\mathbb{T}^n=Z_A$, for some regular $A\in\mathfrak{sl}(2n,\mathbb{R})$ that we choose and fix. The quotient by this action, the total space of the bundle
\begin{gather*}
\bigslant{{\rm Sp}(2n,\mathbb{R})}{\mathbb{T}^n} \to \bigslant{\operatorname{Sp} M}{\mathbb{T}^n} \to M
\end{gather*}
inherits, by the results of the previous section, a 2-form $p^*\omega+i_*\omega^{\mathbb{T}}_A$ with
\begin{gather*}
\iota_X i_*\omega^F_A=\begin{cases} 0 & X \ \text{horizontal},\\
\iota_X\omega^{\mathbb{T}}_A & X \ \text{vertical}\end{cases}
\end{gather*}
depending on the choice of the connection (which survives to the quotient), but well defined in the vertical direction. This form is not necessarily closed: it would be, provided the principal bundle was flat, or if the holonomy of the original connection was restricted to the torus we divided by; this is not always the case. We point out that the form can be defined for any bundle with a cocycle with values in the symplectomorphisms of the fiber, cf.~\cite{dusa}.

We proceed to enhance our form a little. We recall the construction from \cite{fat}. Consider the transitive symplectic action of the group ${\rm Sp}(2n,\mathbb{R})$ on $\big(\bigslant{{\rm Sp}(2n,\mathbb{R})}{\mathbb{T}^n},\omega^{\mathbb{T}}_A\big)$. Since the quotient is simply-connected (by elementary algebraic topology), it has a well defined moment map $\mu\colon \bigslant{{\rm Sp}(2n,\mathbb{R})}{\mathbb{T}^n}\longrightarrow \mathfrak{sp}(2n,\mathbb{R})^*$, and since $H^2(\mathfrak{sp}(2n,\mathbb{R}))=H^1(\mathfrak{sp}(2n,\mathbb{R}))=0$, this moment map is coadjoint-invariant. The connection's curvature form $\Omega$ survives dividing out the torus and this gives a 2-form $\Omega^{\mathbb{T}}$ on the quotient, with values in $\mathfrak{sl}(2n,\mathbb{R})$. By~\cite{fat} (and references therein), the 2-form $i_*\omega^F_A + \mu\Omega^{\mathbb{T}}$ is closed.

It may fail to be non-degenerate, but only in the horizontal direction. The idea in \cite{fat} tracing back to Thurston's paper~\cite{thurston} is to get rid of the possible degeneracies by adding $\lambda p^*\omega$ for sufficiently large~$\lambda$~-- the form $i_*\omega^F_A + \mu\Omega^{\mathbb{T}} + \lambda p^*\omega$ would be now both closed and non-degenerate. In \cite{fat} this $\lambda$ can be easily chosen because the bundle is compact, which our bundle obviously fails to be. But the horizontal part of the form is appropriately bounded in our case as well: it is invariant by the ${\rm Sp}(2n,\mathbb{R})$ action and thus descends to the (compact) base manifold. The existence of $\lambda$ follows. We summarise this discussion in the following theorem and the remarks afterwards.
\begin{Theorem} \label{main2}
The manifold $\big(\bigslant{\operatorname{Sp} M}{\mathbb{T}^n},i_*\omega^{\mathbb{T}}_A + \mu\Omega^{\mathbb{T}} + \lambda p^*\omega\big)$ is symplectic, for some $\lambda\in\mathbb{R}$.
\end{Theorem}
Of course, we could have opted for small perturbation of $p^*\omega$ choosing the structure $\frac{1}{\lambda}i_*\omega^{\mathbb{T}}_A + \frac{1}{\lambda}\mu\Omega^{\mathbb{T}} + p^*\omega$, but this is inconsequential.
\begin{Remark}\label{dodac}
Observe that $i_*\omega^{\mathbb{T}}_A$ is closed iff the connection is flat (i.e., the original connection has its holonomy in the torus we divide by). For any two horizontal fields $X$ and $Y$, and vertical~$Z$ we have ${\rm d}i_*\omega^{\mathbb{T}}_A(X,Y,Z)=i_*\omega^{\mathbb{T}}_A([X,Y],Z)$ which is equivalent to $[X,Y]$ being horizontal by non-degeneracy in the vertical direction. This in particular implies that $\mu\Omega^{\mathbb{T}}$ is closed iff it is zero, while one might have expected ``$\mu\Omega^{\mathbb{T}}$ proportional to $p^*\omega$'' to be a~meaningful geometric condition.
\end{Remark}
\begin{Remark}\label{dodacdwa}
While the resulting symplectic form is the prescribed form $i_*\omega^{\mathbb{T}}_A$ along each fiber (as described in \cite{fat}) the local product structure is problematic. By the previous remark, the total space of a flat bundle is locally trivialized as a product of symplectic manifolds. We would like to state that these conditions are equivalent, however we see no obvious way to prove it. The problem is that Darboux coordinates in general cannot preserve neither the horizontal nor vertical directions, and we strongly suspect that especially in the non-compact setting there might appear a non-flat connection yielding a symplectic structure ${\rm Sp}(2n,\mathbb{R})$-symplectomorphic to a product. However, we were unable to procure such an example.
\end{Remark}
\begin{Remark}\label{main3}
The construction carries over to the compact case (yielding symplectic $\bigslant{\operatorname{U} M}{\mathbb{T}^n}$ or $\bigslant{\operatorname{SO} M}{\mathbb{T}^n}$). We note that the non-compact bundle gives a parallelizable manifold, while the compact ones may fail to do so (for example, $\bigslant{\operatorname{U} \mathbb{S}^2}{\mathbb{T}^1}=\mathbb{S}^2$).
\end{Remark}
\begin{Remark} \label{sedno} The construction applies to the foliated case as well. Consider a transversally symplectic foliation of codimension $2n$, $(M,\mathcal{F},\omega)$. There exist the canonically lifted transversally symplectic foliation
\begin{gather*}
\Big(\bigslant{\operatorname{Sp} M}{\mathbb{T}^n},\mathcal{F}^{\bigslant{{\rm Sp}}{\mathbb{T}}},i_*\omega^{\mathbb{T}}_A + \mu\Omega^{\mathbb{T}} + \lambda p^*\omega\Big).
\end{gather*}
Assuming the foliation was Riemannian we can reduce the frame bundle further and obtain two other lifted transversally symplectic foliations
\begin{gather*}
\Big(\bigslant{\operatorname{SO} M}{\mathbb{T}^n},\mathcal{F}^{\bigslant{{\rm SO}}{\mathbb{T}}},i_*\omega^{\mathbb{T}}_A + \mu\Omega^{\mathbb{T}} + \lambda p^*\omega\Big),\\
\Big(\bigslant{\operatorname{U} M}{\mathbb{T}^n},\mathcal{F}^{\bigslant{{\rm U}}{\mathbb{T}}},i_*\omega^{\mathbb{T}}_A + \mu\Omega^{\mathbb{T}} + \lambda p^*\omega\Big)
\end{gather*}
on the quotients of the bundles of transverse frames, with transversally symplectic structures defined in the similar fashion as in Theorem~\ref{main2}. The canonical lift of a foliation is described in detail in Proposition~2.4 and Chapter~5 of~\cite{molino}.
\end{Remark}

\subsection[Comparing symplectic cohomology of the base, fiber, and total space of a bundle]{Comparing symplectic cohomology of the base, fiber,\\ and total space of a bundle}

After the canonical lifting is constructed, the proof of Theorem~\ref{aziz} in~\cite{aziztopologiczny} proceeds to compare the basic cohomology of its base and its total space via the Leray--Serre spectral sequence. This seems not to carry over to the symplectic case for reasons that are interesting in their own right. We offer the following two observations, being yet unable to propose a coherent picture for symplectic cohomologies of products and bundles. As is well known, the Leray--Serre spectral sequence assumes the K\"{u}nneth formula for de Rham cohomology.
\begin{Example}\label{disc}
The K\"unneth formula fails for symplectic cohomology in general. For example
\begin{gather*}
H^{i}_{{\rm d}+{\rm d}^{\Lambda}}\big(\mathbb{R}^{2n}\big)=\begin{cases}0 & \text{for $i$ odd}, \\ \mathbb{R} & \text{for $i$ even}, \end{cases}
\qquad \text{and}\qquad
H^{i}_{{\rm d}{\rm d}^{\Lambda}}\big(\mathbb{R}^{2n}\big)=\begin{cases}\mathbb{R} & \text{for $i$ odd}, \\ 0 & \text{for $i$ even}. \end{cases}
\end{gather*}
\end{Example}
\begin{proof}
The space $\mathbb{R}^{2n}$ is endowed with the standard symplectic structure $\omega_0$. Consider a~$({\rm d}+{\rm d}^{\Lambda})$-cycle $x$ in $\Omega^k\big(\mathbb{R}^{2n}\big)$, i.e., ${\rm d}x={\rm d}^{\Lambda}x=0$ (we say it is both closed and \emph{coclosed}). Since $0={\rm d}x$, $x={\rm d}y_1$ by Poincar\'e lemma. Consider ${\rm d}^{\Lambda}y_1$. Since $0={\rm d}^{\Lambda}x={\rm d}^{\Lambda} {\rm d} y_1=-{\rm d} {\rm d}^{\Lambda} y_1$, there is $y_3$ such that ${\rm d}^{\Lambda}y_1={\rm d}y_3$. Proceeding inductively, we can construct a sequence $\{y_{2i-1}\}$ of elements in $\Omega^{k-2i+1}$ satisfying ${\rm d}y_{2i-3}=d^{\Lambda}y_{2i-1}$ and ${\rm d}y_1=x$. This sequence is not unique.

Suppose some $y_{2i-1}$ is coclosed, i.e., ${\rm d}^{\Lambda}y_{2i-1}=0$. Then there is an element $\alpha_i\in \Omega^{k-2i+2}\big(\mathbb{R}^{2n}\big)$ such that ${\rm d}^{\Lambda}\alpha_i=y_{2i-1}$~-- because the Poincar\'e lemma works for the operator ${\rm d}^{\Lambda}$ as well. Then the element $y'_{2i-3}:=y_{2i-3}-{\rm d} \alpha_i$ is coclosed and we still have ${\rm d}y'_{2i-3}={\rm d}^{\Lambda}y_{2i-5}$. Proceeding inductively, we can assume we have chosen a coclosed $y_1$, which then is also \emph{coexact}, $y_1={\rm d}^{\Lambda}z$. Observe that then $x={\rm d}{\rm d}^{\Lambda}z$ and thus represents 0 in $H^{\bullet}_{{\rm d}+{\rm d}^{\Lambda}}$.

It is easy to see that if $k$ is odd, then eventually this procedure will produce a coclosed $y_{k}\in\Omega^{0}\big(\mathbb{R}^{2n}\big)$. If $k$ is even, however, the last $y_{k-1}$ lives in $\Omega^{1}\big(\mathbb{R}^{2n}\big)$ and may be coclosed but need not be. We know however that ${\rm d}{\rm d}^{\Lambda}y_{k-1}=-{\rm d}^{\Lambda}dy_{k-1}=-{\rm d}^{\Lambda}{\rm d}^{\Lambda} y_{k-3}=0$, hence ${\rm d}^{\Lambda}y_{k-1}$ is a constant function $c_x$. It is easy to see that $[x]\mapsto c_x$ gives a monomorphism in even degrees $H^{2k}_{{\rm d}+{\rm d}^{\Lambda}}\big(\mathbb{R}^{2n}\big)\hookrightarrow \mathbb{R}$. We have a natural candidate to help us show that it is also an epimorphism. The form $\omega_0^k$ is easily seen to be a $({\rm d}+{\rm d}^{\Lambda})$-cycle. Observe that
\begin{gather*}
\omega_0^k = \sum_{i_1<\cdots<i_k}{\rm d}x_{i_1}\wedge {\rm d}y_{i_1}\wedge\cdots \wedge {\rm d}x_{i_k}\wedge {\rm d} y_{i_k} \\
\hphantom{\omega_0^k}{} ={\rm d}\sum_{i_1<\cdots<i_k}\sum_{j=1}^{k}\big(\big(x_{i_j}{\rm d} x_{i_1}\wedge {\rm d} y_{i_1}\wedge\cdots\wedge {\rm d}y_{i_j}\wedge\cdots\wedge {\rm d}x_{i_k}\wedge {\rm d} y_{i_k} \\
 \hphantom{\omega_0^k=}{} - y_{i_j}{\rm d}_{i_1}\wedge {\rm d}y_{i_1}\wedge\cdots\wedge {\rm d} x_{i_j}\wedge\cdots\wedge {\rm d}x_{i_k}\wedge {\rm d}y_{i_k} \big) \big)
\end{gather*}
and call the form under the ${\rm d}$ on the right $\alpha_{2k-1}$. In order to compute ${\rm d}^{\Lambda}\alpha_{2k-1}$, observe that $\star_s\alpha_{2k-1}=-\alpha_{2n-2k+1}$. Hence ${\rm d}^{\Lambda} \alpha_{2k-1}=\star_s\big({-}\omega_0^{n-k+1}\big)=-\omega_0^{k-1}$. It follows that the constant~$c_{\omega_0^{k}}$ is $(-1)^k$, giving the epimorphism as claimed. Similar computations prove the assertion about~$H^{\bullet}_{{\rm d}{\rm d}^{\Lambda}}$ with the obvious difference that it is not~$\omega_0^k$ that provides a representative of~1 in each degree: it is instead represented by the antiderivative of these forms. This is notably different behaviour than that present on the compact symplectic manifolds.
\end{proof}

This of course contradicts the K\"unneth formula. Observe that the same reasoning applies to any contractible symplectic manifold, with any symplectic structure. In particular, these cohomologies will not distinguish between exotic symplectic structures on~$\mathbb{R}^{2n}$. Note that this example breaks the dualities between the two symplectic cohomologies established in~\cite{tseng}, even though $\omega^n$ represents a non-zero class (a ``fundamental class'' for the duality pairing) in one of them.
\begin{Remark}
Even having established some analogue of the Leray--Serre spectral sequence for the symplectic cohomology (which the previous proposition shows not to be straightforward) we encounter two other obstacles. The proof in \cite{aziztopologiczny} establishes connections between the basic cohomology of a Riemannian foliation $H^{\bullet}(M,\mathcal{F})$ and invariant basic cohomology of its canonical lifting $H^{\bullet}_{{\rm SO}(n)}\big(\operatorname{SO} M,\mathcal{F}^{{\rm SO}}\big)$, and then between the basic cohomology of the canonical lifting $H^{\bullet}\big(\operatorname{SO} M,\mathcal{F}^{{\rm SO}}\big)$ and cohomology of some compact manifold $H^{\bullet}(W)$. To connect all four there is a well known and crucial isomorphism between the non-invariant de Rham cohomology and its $G$-invariant counterpart for a compact connected~$G$ that boils down to the Stokes theorem. Even using the compact quotient $\bigslant{\operatorname{U} M}{\mathbb{T}^n}$ and the compact group ${\rm U}(n)$ we still lack this tool, since there is no natural candidate for the Stokes theorem analogue for operators~${\rm d}+{\rm d}^{\Lambda}$ and~${\rm d}{\rm d}^{\Lambda}$. This may indicate that $G$-invariant symplectic cohomology will exhibit some interesting behaviour.
\end{Remark}

\section{Proof of the finiteness theorem} \label{wreszciedowod}

In this part we give the promised proof of Theorem \ref{main}. We follow closely \cite{tseng} in proving the following.
\begin{Theorem}
Let $(M,\mathcal{F},g,\omega)$ be a codimension $2n$ Riemannian foliation with a transverse symplectic structure. Then $H^{\bullet}_{{\rm d}+{\rm d}^{\Lambda}}(M\slash\mathcal{F})$ and $H^{\bullet}_{{\rm d}{\rm d}^{\Lambda}}(M\slash\mathcal{F})$ are finite dimensional.
\end{Theorem}
\begin{proof}
We concentrate only on the $({\rm d}+{\rm d}^{\Lambda})$-cohomology, since the proof for the other one follows along the same lines.

Without loss of generality we can have a compatible triple $(M,\mathcal{F},g,J,\omega)$ where we perhaps changed the metric but left $\omega$ intact (and the foliation is still Riemannian). Both~$g$ and~$\omega$ induce star operators via non-degenerate pairings as explained in Section~\ref{preliminaries}. We denote them~$\star_{m}$ and~$\star_{s}$, respectively.

The propositions below follow by exactly the same arguments as in~\cite{tseng}. Denote the formal adjoint of an operator $P$ by $P^*$, $\langle Px,y\rangle=\langle x,P^*y\rangle$. We proceed as follows:
\begin{enumerate}\itemsep=0pt
\item[1)] the operator
\begin{gather*}
D_{{\rm d}+{\rm d}^{\Lambda}}= \big({\rm d}{\rm d}^{\Lambda}\big)\big({\rm d}{\rm d}^{\Lambda}\big)^* + \big({\rm d}{\rm d}^{\Lambda}\big)^*\big({\rm d}{\rm d}^{\Lambda}\big)+{\rm d}^*{\rm d}^{\Lambda}{\rm d}^{\Lambda}{}^*{\rm d} + {\rm d}^{\Lambda}{}^*{\rm d}{\rm d}^*{\rm d}^{\Lambda} + \big({\rm d}^*{\rm d} +{\rm d}^{\Lambda}{}^*{\rm d}^{\Lambda}\big)
\end{gather*}
is a self-adjoint elliptic operator;
\item[2)] therefore
\begin{gather*}
\Omega^k(M\slash \mathcal{F})= \ker D_{{\rm d}+{\rm d}^{\Lambda}}\oplus {\rm d}{\rm d}^{\Lambda}\Omega^k(M\slash\mathcal{F}) \oplus \big({\rm d}^*\Omega^{k-1}(M\slash\mathcal{F})+{\rm d}^{\Lambda}{}^*\omega^{k-1}(M\slash\mathcal{F})\big)
\end{gather*}
with the kernel being of finite dimension;
\item[3)] $H_{{\rm d}+{\rm d}^{\Lambda}}(M\slash\mathcal{F})$ is isomorphic with $\ker D_{{\rm d}+{\rm d}^{\Lambda}}$, hence of finite dimension.
\end{enumerate}
For 1), the self-adjointness is obvious. As for the ellipticity, in order to show that the principal symbol is an isomorphism, we must compute how the adjoints of the operators involved behave. According to \cite[Proposition~3.3]{tseng} and \cite[Proposition~3.1]{pawel}, we have the following equalities and equivalence of the principal symbols
\begin{gather*}
{\rm d}^*= \pm \star_{m}{\rm d} - \kappa\wedge\star_{m},\\
{\rm d}^{\Lambda}=({\rm d}^{c})^*\simeq \sqrt{-1}(\partial -\bar{\partial})^*=\pm(\star_m\partial - \star_{m}\bar{\partial} -\kappa_1\wedge \star_{m} + \kappa_2\wedge \star_{m}),
\end{gather*}
where $\kappa$, $\kappa_1$, and $\kappa_2$ are the various parts of the mean curvature form and the operators $\partial$, $\bar{\partial}$, and ${\rm d}^{c}$ come from the compatible almost complex structure. We conclude that the only factors contributing to the principal symbol of $D_{{\rm d}+{\rm d}^{\Lambda}}$ are $\partial$ and $\bar{\partial}$, from which the ellipticity follows precisely as in \cite[Theorem~3.5]{tseng}. This proves the decomposition in~2). The isomorphism in~3) follows (again, exactly as in~\cite{tseng}): $\ker ({\rm d}+{\rm d}^{\Lambda})=\ker D_{{\rm d}+{\rm d}^{\Lambda}}\oplus {\rm d}{\rm d}^{\Lambda}\Omega^k(M\slash\mathcal{F})$ because, should any ${\rm d}^*A+{\rm d}^{\Lambda}{}^*B$ be $({\rm d}+{\rm d}^{\Lambda})$-closed, then
\begin{gather*}
0=\big\langle A,{\rm d}\big({\rm d}^*A+{\rm d}^{\Lambda}{}^*B\big)\big\rangle+\big\langle B,{\rm d}^{\Lambda} \big({\rm d}^*A+{\rm d}^{\Lambda}{}^*B\big)\big\rangle \\
\hphantom{0}{} =\big\langle {\rm d}^*A,{\rm d}^*A+{\rm d}^{\Lambda}{}^*B\big\rangle+\big\langle {\rm d}^{\Lambda}{}^*B,{\rm d}^*A+{\rm d}^{\Lambda}{}^*B\big\rangle =\big\|{\rm d}^*A+{\rm d}^{\Lambda}{}^*B\big\|^2.
\end{gather*}
Thus each class in $H^{\bullet}_{{\rm d}+{\rm d}^{\Lambda}}(M\slash\mathcal{F})$ has a unique representative in $\ker D_{{\rm d}+{\rm d}^{\Lambda}}$.
\end{proof}

\section{Final remarks}

While the methods to prove Theorem \ref{main} for $H^{\bullet}_{{\rm d}+{\rm d}^{\Lambda}}(M\slash\mathcal{F})$ and $H^{\bullet}_{{\rm d}{\rm d}^{\Lambda}}(M\slash\mathcal{F})$ are the same, we must stress that their roles in the foliated case are not, contrary to the non-foliated theory. Indeed, the dualities of Proposition~3.24 in~\cite{tseng} and Theorem~2.16 in~\cite{angela} give in the compact case the inequalities
\begin{gather*}
\dim H^{\bullet}_{{\rm dR}}(M)\leq \dim H^{\bullet}_{{\rm d}+{\rm d}^{\Lambda}}(M),\qquad \dim H^{\bullet}_{{\rm dR}}(M)\leq \dim H^{\bullet}_{{\rm d}{\rm d}^{\Lambda}}(M),
\end{gather*}
while substantially less is true in the foliated case
\begin{gather*}
\dim H^{\bullet}_{{\rm dR}}(M\slash\mathcal{F})\leq \dim H^{\bullet}_{{\rm d}+{\rm d}^{\Lambda}}(M\slash\mathcal{F})+\dim H^{\bullet}_{{\rm d}{\rm d}^{\Lambda}}(M\slash\mathcal{F})
\end{gather*}
as Lemma 1 and the example of Section 4 in \cite{pawelija} show. We recall it here.
\begin{Example}\label{przy}
Consider the transversally symplectic dimension one foliation $\mathcal{F}_L$ given by the suspension of the matrix $L=\left[\begin{smallmatrix}1 & 1\\ 0 & 1\end{smallmatrix}\right]$ acting on the torus, i.e., \smash{$\bigslant{\mathbb{T}^2\times [0,1]}{(t,0)\sim (Lt,1)}=M$}. The transverse manifold~$T$ of Definition~\ref{fol} is the torus $\mathbb{T}^2$ and the pseudogroup $\Gamma$ is the cyclic group generated by $L$. One can check that $H^0(M\slash\mathcal{F}_L)=H^1(M\slash\mathcal{F}_L)=\mathbb{R}$ and \smash{$H^2(M\slash\mathcal{F}_L)\simeq\mathcal{C}^{\infty}(\mathbb{S}^1\!,\mathbb{R})$}, the infinite dimensional space of smooth functions on a circle. However, the basic symplectic cohomologies (for a standard symplectic structure on~$\mathbb{T}^2$, invariant by~$L$) turn out to be
\begin{enumerate}\itemsep=0pt
\item[1)] $H^{0}_{{\rm d}+{\rm d}^{\Lambda}}(M\slash\mathcal{F}_L)=H^{2}_{{\rm d}+{\rm d}^{\Lambda}}(M\slash\mathcal{F}_L)=\mathbb{R}$,
\item[2)] $H^{1}_{{\rm d}+{\rm d}^{\Lambda}}(M\slash\mathcal{F}_L)=\mathcal{C}^{\infty}\big(\mathbb{S}^1\big)$
\end{enumerate}
and
\begin{enumerate}\itemsep=0pt
\item[1)] $H^{0}_{{\rm d}{\rm d}^{\Lambda}}(M\slash\mathcal{F}_L)=H^{2}_{{\rm d}{\rm d}^{\Lambda}}(M\slash\mathcal{F}_L)=\mathcal{C}^{\infty}\big(\mathbb{S}^1\big)$,
\item[2)] $H^{1}_{{\rm d}{\rm d}^{\Lambda}}(M\slash\mathcal{F}_L)=\mathbb{R}$.
\end{enumerate}
This example suggests that one should not hope to extend finiteness results on symplectic cohomologies to any class of symplectic foliations broader than Riemannian foliations. We note that this dependence on metric structure is somewhat unexpected.
\end{Example}

We note that Theorem \ref{main2} and Remark \ref{main3} applied to the Kodaira--Thurston manifold (cf.~\cite{thurston}) produce, somewhat trivially, another non-K\"ahler symplectic manifold. The mani\-fold~$KT$ is parallelizable, hence the frame bundle is trivial and so is the quotient bundle we constructed, $KT\times\mathbb{S}^2$. Its first Betti number is again 3 and again it does not satisfy the hard Lefschetz theorem. We also note that application to $\mathbb{CP}^2$ yields a symplectic structure satisfying the hard Lefschetz theorem: by elementary algebraic topology the bundle
\begin{gather*}
\mathbb{S}^2 \to \bigslant{\operatorname{Sp}\mathbb{CP}^2}{\mathbb{T}^2} \to \mathbb{CP}^2
\end{gather*}
has cohomology ring isomorphic to $H^{\bullet}\big(\mathbb{S}^2\big)\otimes H\big(\mathbb{CP}^2\big)$.

\pdfbookmark[1]{References}{ref}
\LastPageEnding


\begin{thebibliography}{99}
\footnotesize\itemsep=0pt

\bibitem{angela}
Angella D., Cohomological aspects in complex non-{K}\"ahler geometry,
 \href{https://doi.org/10.1007/978-3-319-02441-7}{\textit{Lecture Notes in Math.}}, Vol.~2095, Springer, Cham, 2014.

\bibitem{baczekija}
Bak L., Czarnecki A., A remark on the {B}rylinski conjecture for orbifolds,
 \href{https://doi.org/10.1017/S1446788711001455}{\textit{J.~Aust. Math. Soc.}} \textbf{91} (2011), 1--12, \href{https://arxiv.org/abs/1001.2435}{arXiv:1001.2435}.

\bibitem{brylinski}
Brylinski J.-L., A differential complex for {P}oisson manifolds,
 \href{https://doi.org/10.4310/jdg/1214442161}{\textit{J.~Differential Geom.}} \textbf{28} (1988), 93--114.

\bibitem{cavalcanti}
Cavalcanti G., New aspects of the ${\rm dd}^c$-lemma, \href{https://arxiv.org/abs/math.DG/0501406}{math.DG/0501406}.

\bibitem{courter}
Courter R.C., The dimension of maximal commutative subalgebras of~{$K_{n}$},
 \href{https://doi.org/10.1215/S0012-7094-65-03219-9}{\textit{Duke Math.~J.}} \textbf{32} (1965), 225--232.

\bibitem{pawelija}
Czarnecki A., Ra\'zny P., Examples of foliations with infinite dimensional
 special cohomology, \href{https://doi.org/10.1007/s10231-017-0684-x}{\textit{Ann. Mat. Pura Appl.}} \textbf{197} (2018),
 399--409, \href{https://arxiv.org/abs/1705.02216}{arXiv:1705.02216}.

\bibitem{masaupq}
del Olmo M.A., Rodr\'{\i}guez M.A., Winternitz P., Zassenhaus H., Maximal
 abelian subalgebras of {${\rm su}(p,q)$}, in X{VII}th {I}nternational
 {C}olloquium on {G}roup {T}heoretical {M}ethods in {P}hysics
 ({S}ainte-{A}d\`ele, {PQ}, 1988), World Sci. Publ., Teaneck, NJ, 1989,
 401--404.

\bibitem{azizeliptyczny}
El~Kacimi-Alaoui A., Op\'erateurs transversalement elliptiques sur un
 feuilletage riemannien et applications, \textit{Compositio Math.} \textbf{73}
 (1990), 57--106.

\bibitem{aziztopologiczny}
El~Kacimi-Alaoui A., Sergiescu V., Hector G., La cohomologie basique d'un
 feuilletage riemannien est de dimension finie, \href{https://doi.org/10.1007/BF01161658}{\textit{Math.~Z.}} \textbf{188}
 (1985), 593--599.

\bibitem{fernandezibanez}
Fern\'andez M., Ib\'a\~nez R., de~Le\'on M., On a {B}rylinski conjecture for
 compact symplectic manifolds, in Quaternionic structures in mathematics and
 physics ({T}rieste, 1994), Int. Sch. Adv. Stud. (SISSA), Trieste, 1998,
 119--126.

\bibitem{he}
He Z., Odd dimenisonal symplectic manifolds, Ph.D.~Thesis, Massachusetts
 Institute of Technology, 2010, available at \url{http://hdl.handle.net/1721.1/60189}.

\bibitem{masaonc}
Hussin V., Winternitz P., Zassenhaus H., Maximal abelian subalgebras of complex
 orthogonal {L}ie algebras, \href{https://doi.org/10.1016/0024-3795(90)90317-6}{\textit{Linear Algebra Appl.}} \textbf{141} (1990),
 183--220.

\bibitem{tenodcourtera}
Laffey T.J., The minimal dimension of maximal commutative subalgebras of full
 matrix algebras, \href{https://doi.org/10.1016/0024-3795(85)90247-2}{\textit{Linear Algebra Appl.}} \textbf{71} (1985), 199--212.

\bibitem{dusa}
Lalonde F., McDuff D., Symplectic structures on fiber bundles,
 \href{https://doi.org/10.1016/S0040-9383(01)00020-9}{\textit{Topology}} \textbf{42} (2003), 309--347, \href{https://arxiv.org/abs/math.SG/0010275}{math.SG/0010275}.

\bibitem{molino}
Molino P., Riemannian foliations, \href{https://doi.org/10.1007/978-1-4684-8670-4}{\textit{Progress in Mathematics}}, Vol.~73,
 Birkh\"auser Boston, Inc., Boston, MA, 1988.

\bibitem{pak}
Pak H.K., Transversal harmonic theory for transversally symplectic flows,
 \href{https://doi.org/10.1017/S1446788708000190}{\textit{J.~Aust. Math. Soc.}} \textbf{84} (2008), 233--245.

\bibitem{masasp}
Patera J., Winternitz P., Zassenhaus H., Maximal abelian subalgebras of
 real and complex symplectic {L}ie algebras, \href{https://doi.org/10.1063/1.525955}{\textit{J.~Math. Phys.}}
 \textbf{24} (1983), 1973--1985.

\bibitem{pawel}
Ra\'zny P., The {F}r\"olicher-type inequalities of foliations, \href{https://doi.org/10.1016/j.geomphys.2017.01.005}{\textit{J.~Geom.
 Phys.}} \textbf{114} (2017), 593--606, \href{https://arxiv.org/abs/1605.03858}{arXiv:1605.03858}.

\bibitem{sugiura}
Sugiura M., Conjugate classes of {C}artan subalgebras in real semi-simple {L}ie
 algebras, \href{https://doi.org/10.2969/jmsj/01140374}{\textit{J.~Math. Soc. Japan}} \textbf{11} (1959), 374--434,
 {E}rratum, \href{https://doi.org/10.2969/jmsj/02320379}{\textit{J.~Math. Soc.
 Japan}} {\bf 23} (1971), 379--383.

\bibitem{thurston}
Thurston W.P., Some simple examples of symplectic manifolds, \href{https://doi.org/10.2307/2041749}{\textit{Proc.
 Amer. Math. Soc.}} \textbf{55} (1976), 467--468.

\bibitem{tseng}
Tseng L.-S., Yau S.-T., Cohomology and {H}odge theory on symplectic
 manifolds:~{I}, \href{https://doi.org/10.4310/jdg/1349292670}{\textit{J.~Differential Geom.}} \textbf{91} (2012), 383--416,
 \href{https://arxiv.org/abs/0909.5418}{arXiv:0909.5418}.

\bibitem{yau}
Tseng L.-S., Yau S.-T., Cohomology and {H}odge theory on symplectic
 manifolds:~{II}, \href{https://doi.org/10.4310/jdg/1349292671}{\textit{J.~Differential Geom.}} \textbf{91} (2012), 417--443,
 \href{https://arxiv.org/abs/1011.1250}{arXiv:1011.1250}.

\bibitem{fat}
Weinstein A., Fat bundles and symplectic manifolds, \href{https://doi.org/10.1016/0001-8708(80)90035-3}{\textit{Adv. Math.}}
 \textbf{37} (1980), 239--250.

\bibitem{yan}
Yan D., Hodge structure on symplectic manifolds, \href{https://doi.org/10.1006/aima.1996.0034}{\textit{Adv. Math.}}
 \textbf{120} (1996), 143--154.

\end{thebibliography}
\end{document}